\documentclass[10pt]{article}

\usepackage{mathptmx}
\usepackage{wrapfig}
\usepackage{eucal}
\usepackage{amsmath}
\usepackage{amscd}
\usepackage{amssymb}
\usepackage{amsthm}
\usepackage{xspace}
\usepackage[all,tips]{xy}
\usepackage[dvips]{graphicx}
\usepackage{graphics}
\usepackage{epsfig}
\usepackage{verbatim}
\usepackage{syntonly}
\usepackage{hyperref}
\usepackage{amssymb}
\usepackage{color}
\usepackage{url}
\usepackage{overpic}
\usepackage{wrapfig}

\newtheorem{theorem}{Theorem}
\newtheorem{proposition}[theorem]{Proposition}

\newtheorem{corollary}[theorem]{Corollary}
\newtheorem{lemma}[theorem]{Lemma}

\input xy
\xyoption{all}

\DeclareMathOperator{\PSL}{PSL}

\newcommand{\bs}[2]{\ensuremath{B_{#1,#2}}} 
\newcommand{\comp}[2]{\ensuremath{\Omega_{#1,#2}}}
\newcommand{\Alt}{\ensuremath{\text{Alt}}}
\newcommand{\Sym}{\ensuremath{\text{Sym}}}

\usepackage{fancyhdr}

\fancypagestyle{plain}

\begin{document}


\title{\textbf{The topology of local commensurability graphs}}

\author{Khalid Bou-Rabee\thanks{K.B. supported in part by NSF grant
    DMS-1405609}\and Daniel Studenmund\thanks{D.S. supported in part
    by NSF grant DMS-1246989}}
\maketitle


\begin{abstract}
We initiate the study of the $p$-local commensurability graph of a group, where $p$ is a prime.
This graph has vertices consisting of all finite-index subgroups of a group, where an edge is drawn between $A$ and $B$ if $[A : A\cap B]$ and $[B: A\cap B]$ are both powers of $p$.
We show that any component of the $p$-local commensurability graph of a group with all nilpotent finite quotients is complete. Further, this topological criterion characterizes such groups.
In contrast to this result, we show that for any prime $p$ the $p$-local commensurability graph of any large group (e.g. a nonabelian free group or a surface group of genus two or more or, more generally, any virtually special group) has geodesics of arbitrarily long length.
\end{abstract}
\vskip.1in
{\small{\bf keywords:}
\emph{commensurability, nilpotent groups, free groups, very large groups}}

\vskip.1in

Let $G$ be a group and $p$ a prime number. 
Recall that two subgroups $\Delta_1 \leq G$ and $\Delta_2\leq G$ are \emph{commensurable} if $\Delta_1 \cap \Delta_2$ is finite-index in both $\Delta_1$ and $\Delta_1$.
We define the \emph{$p$-local commensurability graph} of $G$ to be the graph with vertices consisting of finite-index subgroups of $G$ where two subgroups $A, B \leq G$ are adjacent if and only if $[A: A \cap B][B : A\cap B]$ is a power of $p$.
We denote this graph by $\Gamma_p(G)$.
For a warm-up example, see Figure \ref{fig:sym3}.

The goal of this paper is to draw algebraic information of $G$ from the topology of $\Gamma_p(G)$.

\begin{theorem} \label{thm:nilpotent}
Let $G$ be a group.
The following are equivalent:
\begin{enumerate}
\item For any prime $p$, every component of $\Gamma_p(G)$ is complete.
\item All of the finite quotients of $G$ are nilpotent.
\end{enumerate}
\end{theorem}

\noindent
The proof of Theorem \ref{thm:nilpotent} is in \S \ref{sec:nilpotent}.
The classification of finite simple groups and the structure theory of solvable groups play important roles in our proofs.
Theorem \ref{thm:nilpotent} applies, for example, to Grigorchuk's
group \cite{MR712546}, which is a 2-group and therefore has only
nilpotent finite quotients.

\begin{figure}[bht]
  \begin{center}
    \includegraphics[width=0.6\textwidth]{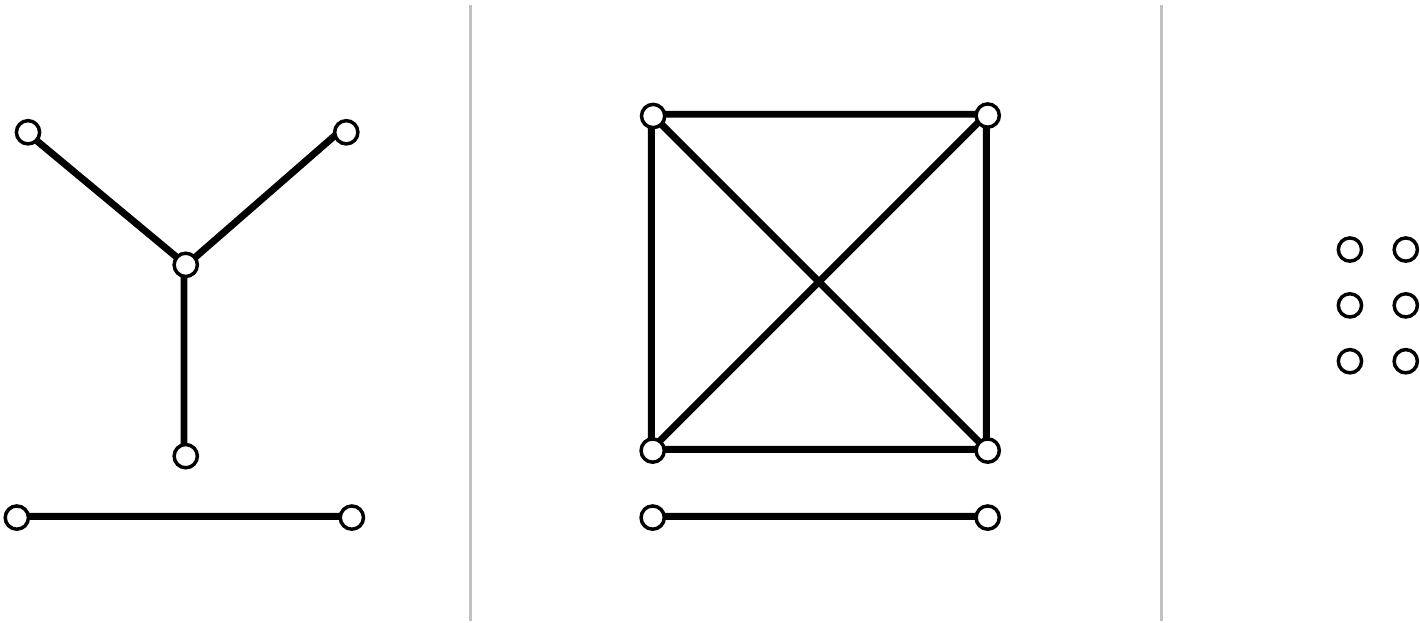}
  \end{center}  \vspace{-10pt}
  \caption{Let $\Sym_3$ be the symmetric group on 3 elements (note $\Sym_3$ is solvable and not nilpotent). The figure above displays $\Gamma_2(\Sym_3), \Gamma_3(\Sym_3)$, and $\Gamma_5(\Sym_3)$ in that order. All $\Gamma_p(\Sym_3)$ for primes $p > 3$ are discrete spaces.}
   \label{fig:sym3}
\end{figure}

In contrast to the above theorem, we show that components of the local commensurability graphs of free groups are far from complete:

\begin{theorem} \label{thm:freegroup}
Let $F$ be a rank two free group.
For any prime $p$ and $N > 0$, there exist infinitely many
geodesics $\gamma$, each in a different component of $\Gamma_p(F)$,
such that the length of each $\gamma$ is greater than $N$.
\end{theorem}

\noindent
We prove Theorem \ref{thm:freegroup} in \S \ref{sec:freegroup}.
A result of Robert Guralnick (which uses the classification of finite simple groups) concerning subgroups of prime power index in a nonabelian finite simple group is used in an essential way in our proof \cite{MR700286}.
Moreover, in our proof we get a clean description of an entire component of the $p$-local commensurability graph of many finite alternating groups. 
See Figure \ref{fig:A7A5p5}, for example.

Our next result demonstrates that arbitrarily long geodesics in the $p$-local commensurability graph of a free group cannot possibly all come from a single component.
We prove this at the end of \S \ref{sec:basic}.

\begin{proposition} \label{prop:freegroupfinitediam}
Let $G$ be any group.
Let $\Omega$ be a connected component of $\Gamma_p(G)$.
Then there exists $C > 0$ such that any two points in $\Omega$ are connected by a path of length less than $C$.
That is, the diameter of $\Omega$ is finite.
Moreover if any vertex of $\Omega$ is a normal subgroup of $G$ then
the diameter of $\Omega$ is bounded above by 3.
\end{proposition}

\noindent
As a consequence of Theorem \ref{thm:freegroup} and Proposition
\ref{prop:freegroupfinitediam}, there exists components of the
$p$-local commensurability graph of a nonabelian free group with no
normal subgroups as vertices (see Corollary \ref{cor:nonormal} at the end of \S \ref{sec:freegroup}). 

Recall that a group is \emph{large} if it contains a finite-index subgroup that admits a surjective homomorphism onto a non-cyclic free group.
Such groups enjoy the conclusion of Theorem \ref{thm:freegroup}. 
See the end of \S \ref{sec:freegroup} for the proof.

\begin{corollary} \label{cor:largegroup}
Let $G$ be a large group.
For any prime $p$ and $N > 0$, there exists infinitely many
geodesics $\gamma$, each in a different component of $\Gamma_p(G)$,
such that the length of each $\gamma$ is greater than $N$.
\end{corollary}

Experiments that led us to the above theorems were done using GAP \cite{GAP4} and Mathematica \cite{MATHEMATICA}.

This paper sits in the broader program of studying infinite groups through their residual properties, which is an area of much activity (see, for instance, \cite{KT2014}, \cite{BK12}, \cite{BM11}, \cite{KG2014}, \cite{BHP14}, \cite{BS13b}, \cite{KM12}, \cite{R12}, \cite{Patel:thesis}, \cite{MR1978431}).
Specifically, a similar object is studied in the recent article \cite{AAHRS2015}. There a graph is constructed with vertices consisting of subgroups of finite index, and an edge is drawn between two vertices if one is a prime-index subgroup (the prime is not fixed) of the other. They show that for every group $G$, their graph is bipartite with girth contained in the set $\{ 4, \infty\}$ and if $G$ is a finite solvable group, then their graph is connected.

\paragraph*{Acknowledgements}
We are grateful to Ben McReynolds and Sean Cleary for useful and stimulating conversations.

\section{Preliminaries and basic facts} \label{sec:basic}

In this section we record some basic facts that will be used
throughout. 
We start with a couple of elementary results.

\begin{lemma} \label{lem:grouptheory}
  Let $\pi: G \to G/N$ be a quotient map. For subgroups $K\leq H \leq
  G$ we have
  \[
  [H : K] = [\pi(H) : \pi(K)] [H\cap N : K\cap N].
  \]
\end{lemma}
\begin{proof}
  We know that
  \[
  [H:K] [K:K\cap N] = [H:K\cap N] \qquad \text{ and } \qquad [H:H\cap
  N] [ H\cap N : K\cap N ] = [ H : K\cap N].
  \]
  Equating left hand sides and rearranging terms yields
  \[
  \frac{ [H : H\cap N] }{ [ K : K\cap N ] } = \frac{ [H : K] }{
    [H\cap N : K\cap N]}.
  \]
  Because $\pi(K) = KN/N = K/(K\cap N)$, and similarly for $H$, we see
  that 
  \[
  [ \pi(H) : \pi(K) ] = \frac{| \pi(H) |}{| \pi(K) |} = \frac{[H : H\cap N]}{[K: K\cap N]}.
  \]
  The desired result follows.
\end{proof}

\begin{lemma} \label{lem:intonenormal}
Let $N$ be a normal subgroup of $G$ and $p$ a prime.
If $A$ and $N$ are both subgroups of index a power of $p$ in $G$,
then $[G: A \cap N]$ is also a power of $p$.
\end{lemma}

\begin{proof}
Let $\pi: G \to G/N$ be the quotient map. Then $[A : A\cap N] =
|\pi(A)|$. Because $G/N$ is a $p$-group, it follows that $[A:A\cap N]$
is a power of $p$. Therefore
$[G: A \cap N] = [G: A][A: A \cap N]$ is a power of $p$.
\end{proof}

Our next couple of lemmas give control of local commensurability graphs under some maps.

\begin{lemma} \label{lem:contraction}
  If $G$ is a group, $\pi: G
  \to Q$ is a surjection, and $\gamma$ a path in $\Gamma_p(G)$, then
  $\pi(\gamma)$ is a path in $\Gamma_q(Q)$ with length bounded above
  by the length of $\gamma$.
\end{lemma}

\begin{proof}
  If $K\leq H \leq G$ then $[ \pi(H) : \pi(K)]$ divides $[ H : K ]$ by
  Lemma \ref{lem:grouptheory}. Therefore adjacent vertices in $\gamma$
  map to adjacent vertices in $\pi(\gamma)$, or are possibly
  identified in $\Gamma_p(Q)$.
\end{proof}

\begin{lemma} \label{lem:embeddings}
  Suppose $G$ is a group and $p$ is prime.
  
  \begin{enumerate}
  \item\label{pt:quotient} If $N$ is a normal subgroup of $G$, then the
  quotient map $\pi:G\to G/N$ induces an isometric graph embedding
  $\Gamma_p(G/N) \to \Gamma_p(G)$ as an induced subgraph.
  \item\label{pt:fi}  If $H$ is a finite-index subgroup of $G$, then the inclusion $i:H\to G$ induces a graph embedding $\Gamma_p(H) \to \Gamma_p(G)$ as an induced subgraph. 
  \item\label{pt:finormal}  If $N$ is a finite-index normal subgroup of $G$, then the inclusion $i:N\to G$ induces an isometric graph embedding $\Gamma_p(N) \to \Gamma_p(G)$ as an induced subgraph. 
  \end{enumerate}
\end{lemma}
\begin{proof}

  For \ref{pt:quotient}, if $\pi : G\to G/N$ is a quotient
  map, then the assignment $K\mapsto \pi^{-1}(K)$ defines a 
  graph embedding $\Gamma_p(G/N) \to \Gamma_p(G)$ whose image is an
  induced subgraph. This embedding is
  isometric by Lemma \ref{lem:contraction}.

  For \ref{pt:fi},  if $H\leq G$ has finite-index, then the assignment
  $K \mapsto i(K)$ defines a graph embedding
  $\Gamma_p(H) \to \Gamma_p(G)$ whose image is an induced subgraph. 
  
  For \ref{pt:finormal}, let $N \lhd G$ be a finite-index subgroup, with assignment $\phi : K \mapsto i(K)$ defined over all subgroups $K$ in $N$.
  Let $H_1, H_2 \in \phi(\Gamma_p(N))$ and let $H_1 = J_1, \ldots, J_n
  = H_2$ be a path in $\Gamma_p(G)$ from $H_1$ to $H_2$. Then for each $i = 1, \ldots, n-1$, we have that 
  $$[J_i : J_i \cap J_{i+1}][J_{i+1} : J_i \cap J_{i+1}]$$
   is a power of $p$.
  By Lemma \ref{lem:contraction}, $\pi(J_1),\dotsc, \pi(J_n)$ is a
  path in $\Gamma_p(G/N)$. Because $J_1\leq N$, this is a path of
  $p$-subgroups of $G/N$. Therefore $[J_i : J_i \cap N]$ is a power
  of $p$  for all $i = 1, \ldots, n$.
  Thus, by Lemma \ref{lem:intonenormal} applied to $J_i \cap N$ and $J_{i+1} \cap J_i$, we have for $i = 1,\ldots, n-1$,
  $$
  [J_i : (J_i \cap N) \cap (J_i \cap J_{i+1})]
  [J_{i+1} : (J_{i+1} \cap N) \cap (J_i \cap J_{i+1})],
  $$
  is a power of $p$.
  Hence, for $i = 1,\ldots, n-1$,
  $$
  [J_{i} : N \cap J_i] [N \cap J_i : N \cap J_i \cap J_{i+1}]
  = [J_i : N \cap J_i \cap J_{i+1}]
  $$
  is a power of $p$ giving that $ [N \cap J_i : N \cap J_i \cap J_{i+1}]$ is a power of $p$, since above we showed that $[J_i : N \cap J_i]$ is a power of $p$.
By a similar argument, we get that $[N \cap J_{i+1} : N \cap J_i \cap J_{i+1}]$ is a power of $p$, and thus $N\cap J_i$ and $N \cap J_{i+1}$ are adjacent in $\Gamma_p(G)$.
It follows that the path $J_1, \ldots, J_n$ can be replaced by the path (which possibly has repeated vertices) $J_1 = J_1 \cap N, J_2 \cap N, \ldots, J_{n-1} \cap N, J_n \cap N = J_n$,
which is entirely contained in $\Gamma_p(H)$.
It follows that $\Gamma_p(H)$ is a geodesic metric space in the path metric induced from $\Gamma_p(G)$, as desired.
\end{proof}

\noindent
Note that the hypothesis of normality in \ref{pt:finormal} cannot be
removed. For example, suppose $S$ and $T$ are disjoint sets with
$|S| = |T| = 5$ and consider the non-normal subgroup
$\Alt_S\times \Alt_T \leq \Alt_{S\cup T}$. It can be shown using Lemma
\ref{lem:conncomp} below that $\Alt_S$ and $\Alt_T$ are in the same
component of $\Gamma_5(\Alt_{S\cup T})$ but in different components of
$\Gamma_5(\Alt_S\times \Alt_T)$.

Our next lemma will lead us to proving our first result concerning free groups.

\begin{lemma} \label{lem:samecomp}
Let $A$ be a vertex in $\Gamma_p(G)$.
Suppose $B$ shares an edge with $A$. If $q^k$ divides $[G:A]$
for some prime $q\neq p$ then $q^k$ divides $[G:B]$.
\end{lemma}

\begin{proof}
In this case, we have 
$$
[G:A \cap B] = [G:A][A:A\cap B] = [G:B][B:A\cap B].
$$
Hence, if $q^k$ divides $[G:A]$, then $q^k$ must divide $[G:B]$ because $[B:A\cap B]$ is a power of $p$.
\end{proof}

\begin{proposition}
The $p$-commensurability graph of a free group has infinitely many components.
\end{proposition}

\begin{proof}
Any free group has subgroups $N_1, N_2, \ldots$ with distinct prime indices $q_1, q_2, \ldots$.
By the previous lemma, any vertex that is in the connected component of $N_i$ has index divisible by $q_i$.
Thus, no path exists between $N_i$ and $N_j$ for distinct $i, j$.
\end{proof}

We finish this section by proving a general result: for any group $G$, any component of $\Gamma_p(G)$ has finite diameter.

\begin{proof}[Proof of Proposition \ref{prop:freegroupfinitediam}]
Let $G$ be any group and $\Omega$ a component of $\Gamma_p(G)$. Take
any vertex $A$ in $\Omega$ and let $N$ be the normal core of $A$. 
Let $\pi : G \to G/N$ be the quotient map. 
Let $D = \{ BN : B \in \Omega \}$. 
We claim that the diameter of $\Gamma_p(G)$ is less than $|D| + 2$.

Let $B$ be a subgroup in $\Omega$. 
Let $V_1, \ldots , V_m$ be a path in $\Gamma_p(G)$ connecting $A$ to $B$.
Then by Lemma \ref{lem:grouptheory}, $\pi(V_1), \ldots, \pi(V_m)$ is a path in $\Gamma_p(G/N)$ connecting $\pi(A)$ to $\pi(B)$.
Hence
$$
\pi(V_1)N, \dotsb, \pi(V_m)N
$$
is a path connecting $A$ to $BN$, and so $BN$ is an element of $\Omega$.
Further, if $[G:B] = n p^{r}$ where $gcd(n, p^r) = 1$, then $[G:BN] = n p^{e}$ by Lemma \ref{lem:samecomp}.
Since $B \leq BN$ and $[G: BN][BN: B] = [G:B]$, we get
$$
n p^e [BN:B] = n p^k
$$
and therefore $[BN:B] = p^{k-e}$.
Hence $BN$ and $B$ are adjacent in $\Gamma_p(G)$. 
It follows that there is an edge from any element in $\Omega$ to one
in $D$, and so the diameter of $\Omega$ is bounded above by the
diameter of the subgraph induced by $D$ plus 2. This gives the desired
bound $|D| + 2$.

If $\Omega$ contains a normal subgroup as a vertex then we can pick $A=N$
in the above argument. Therefore $D$ is the set of $p$-subgroups of
$G/N$. Any two such subgroups are connected by an edge, so the
diameter of $\Omega$ is bounded above by $3$.
\end{proof}

\section{Nilpotent groups: The Proof of Theorem \ref{thm:nilpotent}}

\label{sec:nilpotent}

We will prove Theorem \ref{thm:nilpotent} in two steps, as
Propositions \ref{prop:2implies1} and \ref{prop:1implies2} below. 
For a finite nilpotent group $G$ let $S_p(G)$ denote the
unique Sylow $p$-subgroup of $G$. Recall that $G$ is the direct
product of its Sylow subgroups.

\begin{lemma} \label{lem:subgroupdecomp}
  Suppose $G = S_{p_1}(G) \times \dotsb \times S_{p_k}(G)$ for primes
  $p_1,\dotsb, p_k$. Let $\pi_i : G\to S_{p_i}(G)$ be the quotient
  map for each $i$. Then any subgroup $H\leq G$ has the form
  $H = \pi_1(H)\times \dotsb\times \pi_k(H)$.
\end{lemma}
\begin{proof}
  Choose $\ell_1,\dotsc,\ell_k$ so that $g^{p_i^{\ell_i}} = 1$ for all
  $g\in S_{p_i}(G)$. Choose $N$ so that
  \[
  N p_1^{\ell_1} \dotsb p_{k-1}^{\ell_{k-1}} \equiv 1
  \pmod{p_k^{\ell_k}}.
  \]
  Take any $h\in H$ and write $h=(h_1,\dotsc,h_k)$ for
  $h_i\in S_{p_i}(G)$ for all $i$. Then 
  \[
  h^{N p_1^{\ell_1}\dotsb p_{k-1}^{\ell_{k-1}}} = (1, \dotsc, 1, h_k).
  \]
  Therefore $(1,\dotsc, 1, h_k)\in H$, and so we may identify $\pi_k(H)$
  with a subgroup of $H$. Applying this argument to each other factor, 
  the result follows.
\end{proof}

\begin{proposition} \label{prop:2implies1}
  If $G$ is a finitely generated group such that every finite
  quotient of $G$ is nilpotent, then every component of $\Gamma_p(G)$
  is complete for all $p$.
\end{proposition}
\begin{proof}
  Suppose $A$ and $B$ are subgroups of $G$ in the same component of
  $\Gamma_p(G)$ for some prime $p$ and take any path
  $A = P_0,P_1,\dotsc, P_n = B$ from $A$ to $B$. Let $N$ be a normal,
  finite-index subgroup of $G$ contained in $P_i$ for every $i$. Then
  $G/N$ is a nilpotent group and $\pi(P_0), \pi(P_1),\dotsc, \pi(P_n)$
  is a path in $\Gamma_p(G/N)$, where $\pi: G\to G/N$ is the quotient
  map.

  Let $\mathcal{P}$ be a finite set of primes so that
  $G/N = \prod_{q\in \mathcal{P}} S_q(G/N)$. By Lemma
  \ref{lem:subgroupdecomp} we have decompositions
  $\pi(P_i) = \prod_{q\in\mathcal{P}} S_q(\pi(P_i))$ for each $i$. It
  is straightforward to see that
  \[
  \pi(P_i)\cap \pi(P_{i+1}) = \prod_{q\in\mathcal{P}}
  S_q( \pi(P_i) )\cap S_q( \pi(P_{i+1}) )
  \]
  for any $i$, and so for $j=i$ or $j=i+1$ we have
  \[
  [\pi(P_j) : \pi(P_i) \cap \pi(P_{i+1}) ] = \prod_{q\in\mathcal{P}}
  [S_q( \pi(P_j) ) : S_q( \pi(P_i) )\cap S_q( \pi(P_{i+1}) ) ].
  \]
  Since $\pi(P_i)$ and $\pi(P_{i+1})$ are adjacent in the $p$-local
  commensurability graph of $G/N$, it follows that
  $S_q( \pi(P_i) ) = S_q( \pi(P_{i+1}) )$ for all $i$ and all $q\neq
  p$. Therefore $S_q( \pi(A) ) = S_q( \pi(B) )$ for all $q\neq p$, and
  so $[\pi(A) : \pi(A) \cap \pi(B) ][ \pi(B) : \pi(A) \cap \pi(B) ]$ is
  a power of $p$. Because $[K:L] = [\pi(K) : \pi(L)]$ for any subgroups
  $L\leq K \leq G$ containing $N$, this shows that $A$ and $B$ are
  adjacent in $\Gamma_p(G)$. 
\end{proof}

\begin{lemma} \label{lem:solvablenc}
  If $Q$ is a finite solvable group that is not nilpotent then there
  is some prime $p$ so that a connected component of $\Gamma_p(Q)$ is not complete.
\end{lemma}
\begin{proof}
  Let $\Pi$ be the set of prime divisors of the order of the finite
  solvable group $Q$. For any prime $q\in \Pi$ there is a Hall
  subgroup $H_q$ so that $[Q : H_q] = q^k$ for some $k$ and $q$ does
  not divide the order of $H_q$. Because $Q$ is not nilpotent, there
  is some prime $p$ and a Hall subgroup $H_p$ so that $g^{-1} H_p g
  \neq H_p$ for some $g\in Q$. Then both $H_p$ and $g^{-1} H_p g$ are
  adjacent to $Q$ in $\Gamma_p(Q)$, but there is no edge between $H_p$
  and $g^{-1} H_p g$ in $\Gamma_p(Q)$.
\end{proof}

\begin{lemma} \label{lem:simplecontainssolvable}
  If $Q$ is a non-abelian finite simple group then $Q$ contains a
  non-nilpotent solvable subgroup.
\end{lemma}

\begin{proof} 
By Theorem 1 in  \cite{MR1485492} any non-abelian finite simple group contains a \emph{minimal simple group}, a non-abelian simple group all of whose proper subgroups are solvable.
By the Main Theorem of \cite{MR0369512}, minimal simple groups come from the following list: 
$$
PSL_2(q),\ Sz(q),\ PSL_3(3),\ M_{11},\ ^2F_4(2), \text{ and } \Alt_7.
$$

We show that each of these subgroups contains a solvable and non-nilpotent subgroup:

\begin{enumerate}
\item The groups $PSL_2(2) \cong SL_2(2) \cong \Sym_3 $ and $PSL_2(3) \cong \Alt_4$ are non-nilpotent and solvable.
Suppose now that $q -1 > 2$. As the group of units in a finite field of order $q$ form a cyclic subgroup of order $q-1$, there exists $a$ such that $a^2 \neq 1$. Hence, there exists $a,b$ in any finite field of order $q$ such that $ab = 1$ and $a \neq b$.
The group in $SL_2(q)$ generated by
$$
A:=\begin{pmatrix}
a & 0\\
0 & b
\end{pmatrix} \text{ and }
B:=\begin{pmatrix}
1 & 1\\
0 & 1
\end{pmatrix},
$$
is solvable. Moreover, the order of $A$ divides $(q-1)$ and the order of $B$ divides $q$. Moreover, we have 
$$
A B A^{-1} B^{-1} = 
\begin{pmatrix}
1 & a/b-1 \\
 0 & 1
\end{pmatrix},
$$
which is not central (the only central elements in $SL_2(q)$ are diagonal matrices).
Hence, $\left< A, B \right>$ has non-nilpotent image in $PSL_2(q)$, as elements of coprime order in a finite nilpotent group must commute.

\item Any Suzuki group by \cite[4.]{MR0120283} has a cyclic subgroup that is maximal nilpotent and of index 4 in its normalizer. Because every group of order 4 is nilpotent, this normalizer is solvable. Since the cyclic subgroup is maximal nilpotent, this normalizer cannot be nilpotent.

\item $PSL_3(3) \cong SL_3(3)$ contains $SL_2(3)$, which is solvable and not nilpotent.

\item $M_{11}$ contains $\Sym_5$ as a maximal subgroup by \cite{MR604629} and hence contains $\Sym_3$.

\item $^2F_4(2)$ contains $\PSL_2(25)$ as a maximal subgroup by
  \cite{MR604629} and hence contains a subgroup that is solvable and
  not nilpotent by the above.

\item $\Alt_7$ contains $\Alt_4$.
\end{enumerate}
\end{proof}

\begin{proposition} \label{prop:1implies2}
  Suppose $G$ is a finitely generated group with a finite-index,
  normal subgroup $N$ such that $G/N$ is not nilpotent. Then there is
  some $p$ so that a component of $\Gamma_p(G)$ is not complete.
\end{proposition}
\begin{proof}
  Take $G$ and $N$ as above, let $Q = G/N$ and let $\pi:G\to Q$ be the
  quotient map. If $Q$ is solvable, then by Lemma \ref{lem:solvablenc}
  there is a prime $p$ and subgroups $A,B\leq Q$ in the same component
  of $\Gamma_p(Q)$ that are not adjacent. Then $\pi^{-1}(A)$ and
  $\pi^{-1}(B)$ are non-adjacent vertices in the same component of
  $\Gamma_p(G)$ by Lemma \ref{lem:embeddings}, so $\Gamma_p(G)$ is not
  complete.

  Now consider the case that $Q$ is not solvable. Let
  $Q = N_0 \geq N_1 \geq \dotsb \geq N_{k-1} \geq N_k = \{1\}$ be any
  chain of subgroups such that $N_{i+1}$ is a maximal normal subgroup
  of $N_i$ for all $i$. Because $Q$ is not solvable, there is some $j$
  so that $N_j/N_{j+1}$ is not abelian. Then by Lemma
  \ref{lem:simplecontainssolvable} there is a non-nilpotent solvable
  subgroup $S\leq N_j/N_{j+1}$. By Lemma \ref{lem:solvablenc} there is some
  prime $p$ with a component $\Omega$ of $\Gamma_p(S)$ that is not
  complete. By Lemma \ref{lem:embeddings} the component $\Omega$ fully
  embeds in a component of $\Gamma_p(G)$, which is therefore not
  complete.
\end{proof}

\section{Free groups: The Proof of Theorem \ref{thm:freegroup}}

\label{sec:freegroup}

Let $F$ be the free group of rank two.
Let $p$ be a prime and $N \in \mathbb{N}$ be given. 
By Lemma \ref{lem:embeddings}, to prove Theorem \ref{thm:freegroup} it suffices to find a finite quotient $Q$ of $F$ with subgroups $A, B \leq Q$ such that the length of any geodesic in $\Gamma_p(Q)$ connecting $A$ to $B$ is greater than $N$.
Our candidate for $Q$ is $\Alt_X$, the alternating group on a set $X$ of more than
$p^k > N$ elements, and our candidates for $A$ and $B$ are conjugates of $\Alt_S$ for a subset $S\subseteq X$ with $p^k$ elements.

We first need a couple technical group theoretic results. 
First, we give a description of a connected component in $\Gamma_p(\Alt_X)$. This requires a simple lemma.

\begin{lemma} \label{lem:altoverlap}
If $T_1 \cap T_2$ has more than one element and $|T_1|, |T_2| \geq 4$, then $\left< \Alt_{T_1}, \Alt_{T_2} \right> = \Alt_{T_1 \cup T_2}$.
\end{lemma}
\begin{proof}
We prove this by induction on $|T_1 \cup T_2|$.
The case that $T_1 = T_2$ is clear, so suppose $T_1\neq T_2$.
The base case, when $|T_1| = |T_2| = 4$ and $|T_1 \cap T_2| \in \{2, 3\}$, follows by computation (we did this in \cite{GAP4}).
For the inductive step, suppose without loss of generality that $x\in T_1\setminus T_2$. 
By inductive hypothesis $\left< \Alt_{T_1\setminus \{x\}}, \Alt_{T_2} \right> = \Alt_{T_1\cup T_2 \setminus \{x\}}$.
Arguing similarly if $T_2\setminus T_1$ is nonempty, we reduce to the case when $T_1 \cup T_2 \setminus T_1 \cap T_2$ consists of at most two points.
To finish, we claim that any 3-cycle on points in $T_1\cup T_2$ is in $\left< \Alt_{T_1} , \Alt_{T_2} \right>$.
Let $v_1, v_2, v_3$ be distinct points in $T_1 \cup T_2$.
If $\{v_1, v_2, v_3\} \subseteq T_1$ or $\{v_1, v_2, v_3\} \subseteq T_2$, then we are done.
Thus, by suitably relabeling, we may assume $v_1, v_2 \in T_1$ and $v_3 \in T_2$.
Further, since $T_1 \cup T_2 \setminus T_1 \cap T_2$ consists of at most two points, then by relabeling again, we may assume $v_2 \in T_2$.
Select $w_1, w_2 \in T_1 \cap T_2$ that are distinct from $v_1$, $v_2$, and $v_3$.
Then, by the base case applied to $\Alt_{\{ v_1, v_2, v_3, w_1\}} \leq \Alt_{T_1}$ and $\Alt_{\{ v_1, v_2, v_3, w_2 \}} \leq \Alt_{T_2}$, we obtain that $\Alt_{\{ v_1, v_2, v_3, w_1, w_3 \}}$ is contained in $\left< \Alt_{T_1} , \Alt_{T_2} \right>$, and hence the desired 3-cycle is found.
This completes the proof.
\end{proof}

For any subset $S \subseteq X$, we denote the symmetric group on $S$ by $\Sym_S$ and the alternating group on $S$ by $\Alt_S$. For a subgroup $P \leq \Sym_S$ we define the \emph{support} to be the complement of the fixed point set of the action of $P$ on $S$.

\begin{lemma} \label{lem:decomp}
Let $p$ be a prime number and $k$ an integer so that $p^k > 4$.
Let $X$ be a finite set, $S\subseteq X$, and $P\leq \Sym_X$ a $p$-group with 
support disjoint from $S$.
Let $E$ be an index $p^j$ subgroup of $\Alt_S \times P$.
If $|S| = p^k$ or $|S| = p^k -1$, then we have the decomposition $E = \Alt_T \times
P'$ for some $P' \leq P$ and some
$T \subseteq S$ with $|T| = p^k$ or $|T| = p^k-1$.
\end{lemma}

\begin{proof}
Let $\pi : \Alt_S \times P \to \Alt_S$ be the projection map.
By Lemma \ref{lem:grouptheory} we have
$$
[\Alt_S \times P : E] = [\Alt_S : \pi(E)][1 \times P : E \cap (1 \times P)].
$$
The left hand side of this equation is a power of $p$, so
$[\Alt_S : \pi(E)]$ is a power of $p$. Because $|S| = p^k$ or
$|S| = p^{k}-1$ by assumption, Theorem 1(a) in \cite{MR700286}
immediately implies that either $\pi(E) = \Alt_S$ or $|S| = p^k$ and
$\pi(E) = \Alt_{S \setminus \{v\}}$ for some $v \in S$. Let $T$ denote the
set such that $\pi(E) = \Alt_T$.  
Let $q$ be $3$ if $p \neq 3$ and $q$ be $2$ if $p = 3$.
For the case $p\neq 3$, recall that $\Alt_T$ is generated by 3-cycles by elementary
properties of alternating groups.
In the case $p=3$, note that $p^k >6$.
Because $\Alt_6$ is generated by an element of order $2$ and one of order
$4$, Lemma \ref{lem:altoverlap} implies that
$\Alt_T$ is generated by elements of order $2$ or $4$ in this case.
Therefore in either case it follows that $\Alt_T$ is generated by elements $g_1, \ldots, g_k$ each with order dividing a power of $q$.
Since $\pi$ maps onto $\Alt_T$, we have that for
each $i = 1, \ldots, k$, there exists $v_i \in P$ such that
$(g_i, v_i) \in E$.  Since $v_i \in P$, we have that the order of
$v_i$ is coprime with $g_i$, hence as $q \neq p$, there exists $\ell$ such that
$$
(g_i, v_i)^\ell = (g_i, 1).
$$
It follows then that $E$ contains all of $\Alt_T \times 1$, and hence $E = \Alt_T \times P'$ where $P' \leq P$, as desired.
\end{proof}

Let $\comp{S}{X}$ be the component of $\Gamma_p(\Alt_X)$ containing $\Alt_S$, and
let $\bs{S}{X}$ denote the set of subgroups in $\comp{S}{X}$ isomorphic to $\Alt_T$ for some $|T| \in \{p^k, p^k-1\}$.
For odd primes $p$, we get the following description:

\begin{lemma} \label{lem:conncomp}
Let $S \subseteq X$ be a set of cardinality $p^k$ for some odd prime $p$ such that $p^k > 4$.
Vertices of the component $\comp{S}{X}$ in $\Gamma_p(\Alt_X)$ consist of
two classes of subgroups:
\begin{enumerate}
\item[]{\bf Type 1.} subgroups of the form $\left< \Alt_T, P \right>$, where $|T| = p^k$ and $P \leq \Alt_X$, and
\item[]{\bf Type 2.} subgroups of the form $\left< \Alt_T, P \right>$, where $|T| = p^k-1$ and $P \leq \Alt_X$.
\end{enumerate}
In either case, the subgroup is $\Alt_T \times P$, where $P$ is a $p$-group with support in $T^c$.
Moreover, for all primes $p$, if $V$ is of Type 1 or Type 2, the set $T$ is uniquely determined by $V$.
\end{lemma}

\begin{proof}
We first show uniqueness of $T$. This implies that Type 1 and Type 2 are disjoint classes.
Let $V$ be a vertex with distinct decompositions $\Alt_{T_i} \times P_i$ with
$|T_i| > 3$ and $p$-group $P_i$ with support in $T_i^c$ for $i = 1,2$
such that $T_1 \neq T_2$.
If $T_i \cap T_j$ is empty, then
$$
[V : \Alt_{T_1} \times \Alt_{T_2} \times 1][\Alt_{T_1} \times \Alt_{T_2}\times 1: \Alt_{T_1}\times 1] = [V: \Alt_{T_1} \times 1] = |P_1|,
$$
and thus $[\Alt_{T_1} \times \Alt_{T_2}\times 1: \Alt_{T_1}\times 1] =
|\Alt_{T_2}|$ must be a power of $p$. But this is impossible as
$|\Alt_{T_2}|$ is either $(p^k)!/2$ or $(p^k-1)!/2$ for $p^k > 4$.
Thus, $T_1$ and $T_2$ overlap. If $T_1\neq T_2$ then $\Alt_{T_1}
\times 1$ cannot be normal because $\Alt_{T_2}$ acts transitively on
$T_2$. But $\Alt_{T_1} \times 1$ is clearly normal in $\Alt_{T_1}
\times P_1$, so this is a contradiction.
Therefore $T_1 = T_2$.

Since elements in $\bs{S}{X}$ are of Type 1 or 2, it suffices to show that any $E$ that is adjacent to an element of Type 1 or 2 must itself be of Type 1 or 2.

Let $E$ be adjacent to $V = \Alt_T \times P$ where $P$ is a $p$-group
with support in $T^c$ and $|T| = p^k$ or $|T| = p^k-1$.
Then $E \cap V$ is a subgroup of $\Alt_T \times P$ of index a power of $p$.
By Lemma \ref{lem:decomp}, $E \cap V = \Alt_T \times P'$ or $E\cap V =
\Alt_{T \setminus \{ v \}}\times P'$ where $P' \leq P$ and $v \in T$.
We will therefore assume without loss of generality that $E$ contains $\Alt_T \times 1 = \Alt_T$ as a subgroup of $p$ power index.

Suppose that $E$ does not leave $T$ invariant.
Let $T_1, T_2, \cdots, T_k$ be the orbit of $E$ acting on $T$ and note
that $E$ contains $\Alt_{T_i}$ for each $i$.
Suppose $T_i \cap T_{i+1}$ has fewer than two elements for some $i$. 
The group $\Alt_{T_{i}}$ contains
$\Alt_{T_{i}\setminus T_i \cap T_{i+1}}$, which includes a permutation of
order $2$ since $|T_i| > 4$.
Hence $E$ contains $\Alt_{T_i} \times \mathbb{Z}/2\mathbb{Z} \geq
\Alt_{T_i}$.
This is impossible, as $\Alt_{T_i}$ is of index $p^k$ in $E$ for an odd
prime $p$.
We therefore know that $T_i \cap T_{i+1}$ has more than two elements for
every $i$. Then by applying Lemma \ref{lem:altoverlap} we conclude that 
$E$ contains $\Alt_{T_1 \cup T_2 \cup \cdots \cup T_k}$.
Since $E$ contains $\Alt_T$ as a subgroup of prime power index and $T_1
\cup \cdots \cup T_k \neq T_1$, it follows that $|T_1 \cup \cdots \cup
T_k| = p^k$ and in fact $E$ contains $\Alt_{T_1 \cup \cdots \cup T_k}$ as
a subgroup of index $p^\ell$ for some $\ell$.

We may therefore assume, after replacing $T$ with $T_1\cup \dotsb \cup
T_k$ if necessary, that $E$ leaves $T$ invariant. 
Then $E \leq \Sym_T \times Q$ where $Q$ is a group with support disjoint from $T$.
Let $\pi : \Sym_T \times Q \to \Sym_T$ be the projection onto the first coordinate.
By Lemma \ref{lem:grouptheory}, $[\pi(E):\Alt_T]$ divides $[E:\Alt_T]$ and hence is a power of $p$.
It follows that $\pi(E) = \Alt_T$, as $\Alt_T$ is a maximal subgroup of $\Sym_T$ of index two.
Further, since $\Alt_T$ is normal, we apply Lemma \ref{lem:grouptheory} to the map $\psi : \Alt_T \times Q \to Q$ to see that $|\psi(E)|$ is a power of $p$.
Applying Lemma \ref{lem:decomp} we obtain the desired conclusion.
\end{proof}

The prime $p = 2$ requires relaxing the conclusion of Lemma \ref{lem:conncomp}, since any symmetric group on three or more elements contains an alternating group of index 2.
\begin{lemma} \label{lem:conncompeven}
Let $S \subseteq X$ be a set of cardinality $2^k$ such that $k > 2$.
Vertices of the component $\comp{S}{X}$ in $\Gamma_p(\Alt_X)$ is at least one of two types:
\begin{enumerate}
\item[]{\bf Type 1'.} subgroups $V$ such that $\Alt_T \times 1 \leq V \leq \Sym_T \times P$, where $|T| = 2^k$ and $P \leq \Alt_X$,  and
\item[]{\bf Type 2'.} subgroups $V$ such that $\Alt_T \times 1 \leq V \leq \Sym_T \times P$, where $|T| = 2^k-1$ and $P \leq \Alt_X$.
\end{enumerate}
In either case, $P$ is a $2$-group with support in $T^c$.
\end{lemma}

\begin{proof}
Since elements in $\bs{S}{X}$ are of Type 1' or 2', it suffices to show that any $E$ that is adjacent to an element of one of the types must itself be of one of the types.

Let $E$ be adjacent to some $V$ with $\Alt_T \times 1 \leq V \leq
\Sym_T \times P$ where $P$ is a $2$-group.
Because $V$ has index a power of $2$ in $\Sym_T\times P$, we know that
$E\cap V$ also has index a power of $2$ in $\Sym_T\times P$.
Since $\Alt_T \times P$ is a normal subgroup of $\Sym_T \times P$, we have by Lemma \ref{lem:intonenormal} that $(\Alt_T \times P) \cap E \cap V$ has index a power of $2$ in $\Sym_T \times P$, and hence in $\Alt_T\times P$.
By Lemma \ref{lem:decomp}, $E \cap V \cap (\Alt_T \times P) = \Alt_T \times
P'$ or $E\cap V \cap (\Alt_T\times P) = \Alt_{T \setminus \{ v \}}\times P'$ where $P' \leq P$ and $v \in T$.
We conclude that $\Alt_T \times 1$ or $\Alt_{T \setminus \{ v \}}
\times 1$  has index a power of $2$ in $E\cap V \cap(\Alt_T\times P)$,
and hence has index a power of $2$ in $E$.
We will therefore assume without loss of generality that $E$ contains
$\Alt_T \times 1 = \Alt_T$ as a subgroup with index a power of $2$, where $|T| = 2^k$ or $|T| = 2^k -1$.

Suppose that $E$ does not leave $T$ invariant.
Let $T_1, T_2, \cdots, T_k$ be the orbit of $E$ acting on $T$ and note
that $E$ contains $\Alt_{T_i}$ for each $i$.
Suppose $T_i \cap T_{i+1}$ has fewer than two elements for some $i$. 
The group $\Alt_{T_{i}}$ contains
$\Alt_{T_{i}\setminus T_i \cap T_{i+1}}$, which includes a permutation of
order $3$ because $|T_i| > 3$.
Hence $E$ contains $\Alt_{T_i} \times \mathbb{Z}/3\mathbb{Z} \geq
\Alt_{T_i}$.
This is impossible, as $\Alt_{T_i}$ is of $2$ power index in $E$.
We therefore know that $T_i \cap T_{i+1}$ has more than two elements for
every $i$. Then by applying Lemma \ref{lem:altoverlap} we conclude that 
$E$ contains $\Alt_{T_1 \cup T_2 \cup \cdots \cup T_k}$.
Since $E$ contains $\Alt_T$ as a subgroup of prime power index and $T_1
\cup \cdots \cup T_k \neq T_1$, it follows that $|T_1 \cup \cdots \cup
T_k| = 2^k$ and in fact $E$ contains $\Alt_{T_1 \cup \cdots \cup T_k}$ as
a subgroup of index $2^\ell$ for some $\ell$.

We may therefore assume, after replacing $T$ with $T_1\cup \dotsb \cup
T_k$ if necessary, that $E$ leaves $T$ invariant. 
Then $\Alt_T \times 1 \leq E \leq \Sym_T \times Q$ where $Q$ is a 2-group with support disjoint from $T$, as desired.
\end{proof}

Note that groups of Type 1 and Type 2 are of Type 1' and Type 2' respectively.
The next result allows us to restrict attention to geodesics in $\bs{S}{X}$ when computing distances there.

\begin{figure}[t]
\begin{overpic}[width=1\textwidth]{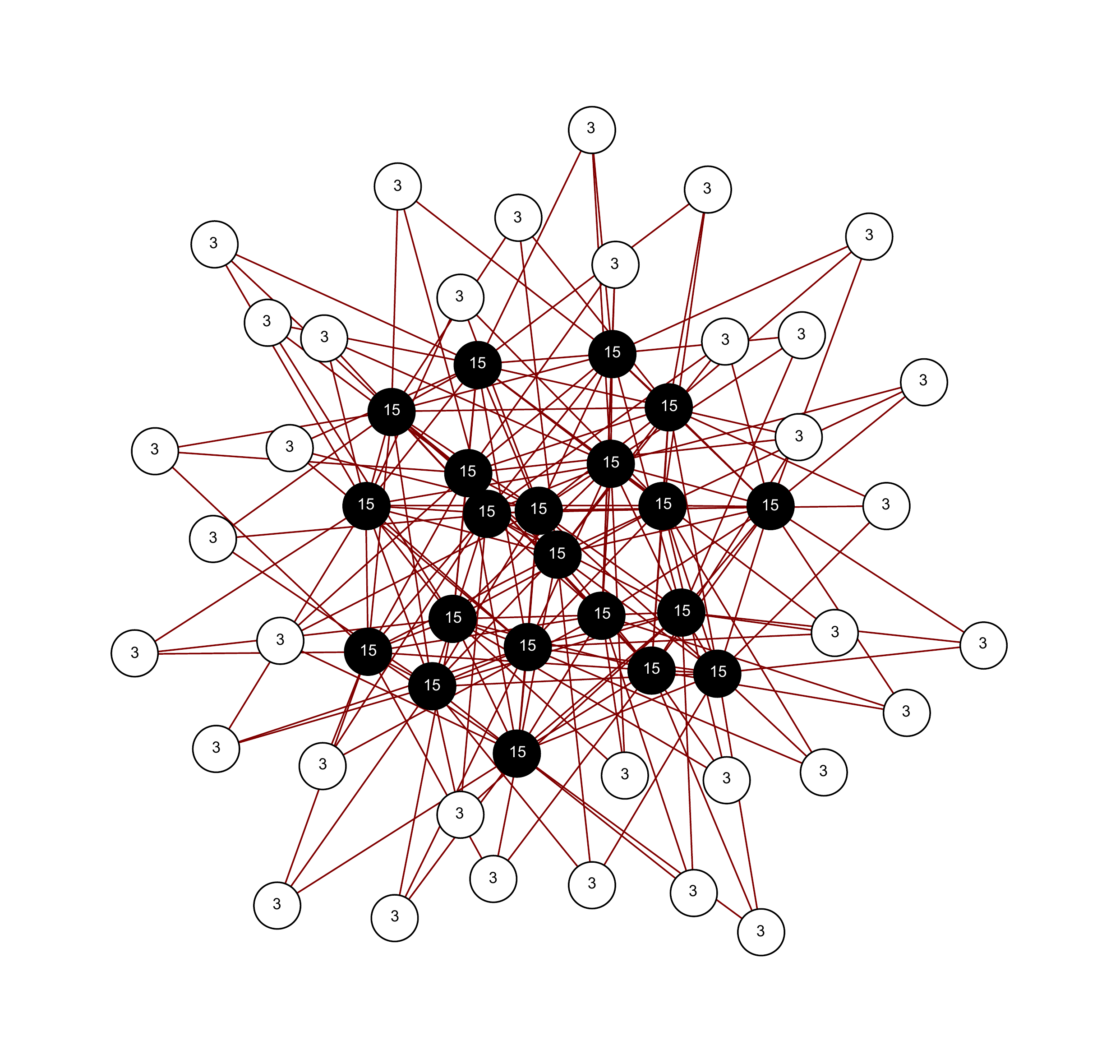}
\end{overpic}
\caption{$\comp{S}{X}$ with $|S| = 5$ and $|X| = 7$. The coloring gives the types and the numbers give the valence of each vertex.
This figure was generated using GAP \cite{GAP4} and Mathematica \cite{MATHEMATICA}}\label{fig:A7A5p5}
\end{figure}

\begin{lemma} \label{lem:Bgeodesic}
Let $S \subseteq X$ be a set of cardinality $p^k > 4$ for some prime 
$p$ and integer $k$.
Then $\bs{S}{X}$ is a geodesic metric space in the path metric induced from $\comp{S}{X}$.
\end{lemma}

\begin{proof}
We first need a local fact.
Let $V_1,V_2$ be two adjacent vectors in $\comp{S}{X}$.
If $p$ is odd, then by Lemma \ref{lem:conncomp} we have $V_i =
\Alt_{T_i} \times P_i$, where $|T_i| = p^k$ or $|T_i| = p^k-1$ and the support of $P_i$ is disjoint from $T_i$ for $i = 1,2$.
If $p = 2$, then by Lemma \ref{lem:conncompeven}, $\Alt_{T_i} \times 1
\leq V_i \leq \Sym_{T_i} \times P_i$, where $|T_i| = p^k$ or $|T_i| = p^k-1$ and the support of $P_i$ is disjoint from $T_i$ for $i = 1,2$.
We claim that in either case, $\Alt_{T_1}$ and $\Alt_{T_2}$ are connected by an edge in
$\Gamma_p(\Alt_X)$.

Since $V_1$ and $V_2$ are adjacent, we have that
$$
[V_1 : V_1 \cap V_2] [V_2 : V_1 \cap V_2]
$$
is a power of $p$.
Thus, when $p$ is odd, Lemma \ref{lem:decomp} applied twice along with
the uniqueness in Lemma \ref{lem:conncomp} gives that $V_1 \cap V_2$
is $\Alt_{S} \times P$ where $S \subseteq T_1 \cap T_2$ satisfies $|S| = p^k$ or $|S| = p^k-1$ and $P \leq P_1 \cap P_2$. Thus it is straightforward to see that $\Alt_{T_1}$ is adjacent to $\Alt_{T_2}$.

When $p = 2$, set $H_i = \Alt_{T_i} \times 1$ and $\Lambda = V_1 \cap V_2$.
Then $H_i$ is normal in $V_i$, thus $H_i \cap \Lambda$ is normal in $\Lambda$.
Since $[\Sym_{T_i} \times P_i : H_i]$ is a power of 2 and 
$$[\Sym_{T_i}\times P_i :V_i][V_i:H_i] = [\Sym_{T_i} \times P_i: H_i],$$
we get $[V_i:H_i]$ is a power of 2. Since $H_i$ is normal in $V_i$, Lemma \ref{lem:intonenormal} implies that $[V_i:H_i \cap \Lambda]$ is a power of 2. Further, as $[V_i:\Lambda]$ is a power of 2 and
$$
[V_i: \Lambda][\Lambda : H_i \cap \Lambda] = [V_i : H_i \cap \Lambda ]
$$
we conclude that $[\Lambda:H_i \cap\Lambda]$ is a power of 2 for $i = 1,2$.
Thus, applying Lemma \ref{lem:intonenormal} to $H_1 \cap \Lambda \lhd \Lambda$ and $H_2 \cap \Lambda \lhd \Lambda$, we have that $H_1 \cap H_2 \cap \Lambda$ has index a power of $2$ in $\Lambda$.
As
$$[V_i : \Lambda][\Lambda:H_1 \cap H_2 \cap \Lambda] = [V_i : H_1 \cap H_2 \cap \Lambda],$$
it follows that $[V_i:H_1 \cap H_2 \cap \Lambda]$ is a power of 2. Because $[V_i:H_i]$ is also a power of $2$ (shown above) and
$$
[V_i: H_i][H_i : H_1 \cap H_2 \cap \Lambda] = [V_i: H_1 \cap H_2 \cap \Lambda]
$$
we have $[H_i:H_1 \cap H_2 \cap \Lambda]$ is a power of $2$ for
each $i$.
By applying Theorem 1(a) in \cite{MR700286} and the uniqueness in Lemma \ref{lem:conncomp}, we have $H_1 \cap H_2 \cap
\Lambda$ is $\Alt_{S}$ for some $S \subseteq
T_1\cap T_2$ with $|S| = p^k$ or $|S| = p^k -1$.
Thus, $\Alt_{T_1}$ is adjacent to $\Alt_{T_2}$, as claimed.

Now let $\gamma$ be a path in $\comp{S}{X}$ that, except for its endpoints, is entirely in the complement of $\bs{S}{X}$.
Enumerate the vertices of $\gamma$ in the order they are traversed,
$$
V_1, V_2, \ldots, V_m, \text{ where } \Alt_{T_i} \times 1 \leq V_i \leq \Sym_{T_i} \times P_i \text{ for all } i = 1, \ldots, m
$$
Then by the previous claim, we may form a new path (after throwing out repeated vertices)
$$
\Alt_{T_{i_1}} , \Alt_{T_{i_2}}, \ldots, \Alt_{T_{i_n}}.
$$
that is entirely contained in $\bs{S}{X}$ and has the same endpoints as $\gamma$.
It follows that $\bs{S}{X}$ is geodesic in $\comp{S}{X}$, as desired.
\end{proof}

\begin{proposition} \label{prop:longpaths}
Let $S \subseteq X$ be a set of cardinality $p^k > 4$ for some prime $p$
and integer $k$.
There exists $V,W \in \bs{S}{X}$ such that any path in $\comp{S}{X}$ connecting $V$ to $W$ has length at least $p^k - \max\{ 0, 2p^k - |X|\}$.
\end{proposition}

\begin{proof}
By Proposition \ref{lem:Bgeodesic}, it suffices to show that there exists $V,W \in \bs{S}{X}$ such that any path in $\bs{S}{X}$ has length greater than $|X|-p^k$.
Let $O_1,O_2 \subseteq X$ with $|O_1 \cap O_2| \leq \max\{ 0,  2p^k-|X| \}$.
Let $E_1, E_2, \ldots, E_m$ be distinct vertices in a non-back-tracking path in $\bs{S}{X}$ connecting $\Alt_{O_1}$ to $\Alt_{O_2}$.
Let $T_1, T_2, \ldots, T_m$ be subsets of $X$ such that $E_i = \Alt_{T_i}$ for $i =1, \ldots, m$.
For each $i = 1, \ldots, m$, we have one of three cases:
\begin{enumerate}
\item $E_i$ is Type 1 and $E_{i+1}$ is Type 1:
In this case, $|T_{i+1} \cap T_i| =  |T_i| - 1 = p^k-1$.

\item $E_i$ is Type 1 and $E_{i+1}$ is Type 2:
In this case, $T_{i+1} \subset T_i$ and $|T_{i+1}| = |T_i| -1 = p^k-1$.

\item $E_i$ is Type 2 and $E_{i+1}$ is Type 1:
In this case, $T_{i+1} \supset T_i$ and $|T_{i+1}| = |T_i| + 1 = p^k$.

\item $E_i$ is Type 2 and $E_{i+1}$ is Type 2:
This case never occurs, as $[\Alt_T : \Alt_U]$ is not a power of $p$ for
any proper subset $U\subset T$ with $|T|=p^k-1$.
\end{enumerate} 

\noindent
Thus, we see that for each $i$, we see that $T_i$ and $T_{i+1}$ differ by moving, adding, or removing at most one element.
It follows that $m \geq p^k - |O_1 \cap O_2| \geq p^k - \max\{ 0 , 2p^k - |X| \}.$
\end{proof}

\begin{proof}[Proof of Theorem \ref{thm:freegroup}]
  Let $F$ be a rank two free group and $p$ a prime. Given $N>0$, choose $k$ so that
  $p^k > N$ and $p^k > 4$. For any finite set $X$ with $|X| > 2p^k$, let $\gamma_X$
  be a path of length $p^k$ in $\Gamma_p(\Alt_X)$ guaranteed by
  Proposition \ref{prop:longpaths}. Then pulling back
  $\gamma_X$ over any surjection $\pi : G \to \Alt_X$ produces a path of
  length $p^k$ in $\Gamma_p(F)$ by Lemma \ref{lem:embeddings}. By
  Lemma \ref{lem:samecomp}, sets $X_1$ and $X_2$ with relatively prime
  cardinalities will produce geodesics in different components of $\Gamma_p(F)$.
\end{proof}


\begin{proof}[Proof of Corollary \ref{cor:largegroup}]
Let $G$ be a large group, $p$ a prime, and $N > 0$.
Since a finite-index subgroup of a nonabelian free group is nonabelian, there exists a normal finite-index subgroup $H \leq G$ that surjects onto $F$, the free group of rank $2$.
By Lemma \ref{lem:contraction} and Theorem \ref{thm:freegroup}, there exists vertices $V, W \in \Gamma_p(H)$ such that any path connecting them in $\Gamma_p(G)$ has length greater than $N$.
The result now follows from Lemma \ref{lem:embeddings}, as $\Gamma_p(H)$ isometrically embeds into $\Gamma_p(G)$.
\end{proof}

\begin{corollary} \label{cor:nonormal}
Let $G$ be a large group and $p$ be a prime.
There exists a connected component of $\Gamma_p(G)$ that does not contain any normal subgroup.
\end{corollary}
\begin{proof}
  By Proposition \ref{prop:freegroupfinitediam}, any component of $\Gamma_p(G)$
  containing a normal subgroup as a vertex has diameter at most 3. By
  Corollary \ref{cor:largegroup}, there are components of $G$ with
  arbitrarily long geodesics.
\end{proof}

\bibliography{refs}
\bibliographystyle{amsalpha}


\noindent
Khalid Bou-Rabee \\
Department of Mathematics, CCNY CUNY \\
E-mail: khalid.math@gmail.com \\

\noindent
Daniel Studenmund \\
Department of Mathematics, University of Utah \\
E-mail: dhs@math.utah.edu \\

\end{document}